\documentclass[11pt]{amsart}
\usepackage{amssymb}
\usepackage{latexsym}
\usepackage{graphicx}
\usepackage{subfigure}
\usepackage{tikz}
\usepackage{hyperref}
\usepackage{ulem}
\usepackage[left=3.8cm, right=3.8cm, top=3cm, bottom=3cm]{geometry}
\usepackage{color}
\usepackage{mathrsfs}

\hypersetup{ colorlinks   = true, 
urlcolor  = blue, 
linkcolor    = blue, 
citecolor   = blue 
}

\def\beq{\begin{equation}}
\def\eeq{\end{equation}}

\newcommand{\Z}{{\mathbb Z}}
\newcommand{\R}{{\mathbb R}}
\newcommand{\Q}{{\mathbb Q}}
\newcommand{\C}{{\mathbb C}}

\newcommand{\T}{{\mathbb T}}
\newcommand{\E}{{\mathbb E}}

\newcommand{\A}{{\mathbb A}}

\newtheorem{theorem}{Theorem}[section]
\newtheorem{remark}[theorem]{Remark}
\newtheorem{lemma}[theorem]{Lemma}
\newtheorem{defi}[theorem]{Definition}
\newtheorem{prop}[theorem]{Proposition}
\newtheorem{corollary}[theorem]{Corollary}

\newtheorem*{thmh}{Main Theorem}
\newtheorem*{parab}{[Parabolicity $\Rightarrow$ ac] Lemma}

\begin{document}

\title{Invariant graphs and spectral type of  Schr\"odinger operators}

\author{Artur Avila}
\address{Universit\"at Z\"urich, Institut f\"ur Mathematik, Winterthurerstrasse 190, CH-8057 Z\"urich, Switzerland, \& IMPA, Estrada Dona Castorina 110, Rio de Janeiro, Brazil}
\email{artur.avila@math.uzh.ch}

\author{Konstantin Khanin}
\address{Department of Mathematics, University of Toronto, Bahen Centre, 40 St. George St, Toronto, ON M5S 2E4 Canada \& IITP RAN, Bolshoy Karetny per. 19, Moscow 127051 Russia}
\email{khanin@math.utoronto.ca}

\author[Martin Leguil]{Martin Leguil$^*$}
\thanks{$^*$M.L. was supported by the ERC project 692925 NUHGD of Sylvain Crovisier.}
\address{Department of Mathematics, University of Toronto, Bahen Centre, 40 St. George St, Toronto, ON M5S 2E4 Canada \& CNRS-Laboratoire de Math\'ematiques d’Orsay, UMR 8628, Universit\'e Paris-Saclay, Orsay Cedex 91405, France}
\email{martin.leguil@math.u-psud.fr}

\begin{abstract} 
In this paper we study spectral properties of Schr\"odinger operators with quasi-periodic potentials related to quasi-periodic action minimizing trajectories for analytic twist maps. We prove that the spectrum contains a component of absolutely continuous spectrum provided that the corresponding trajectory of the twist map belongs to an analytic invariant curve. 
\end{abstract}

\maketitle

\tableofcontents

%

\section{Introduction}

In this paper we discuss connections between the Aubry-Mather theory and the spectral theory of Schr\"odinger operators with quasi-periodic potentials.
The Aubry-Mather theory was developed in the 1980s. In a certain sense it can be considered as a global extension of the KAM theory. It also provides
a variational approach to the KAM phenomenon. In what follows we discuss only the 2D twist map setting  corresponding to Hamiltonian systems with two 
degrees 
of freedom. Namely, let $\T:=\R/\Z$ and consider a family $\{\psi_{\lambda f}\}_{\lambda \in \R}$ of standard type maps on the cylinder:
$$
\psi_{\lambda f}\colon \left\{
\begin{array}{rcl}
\T\times \R &\to& \T\times \R,\\
(\varphi,r) &\mapsto& (\varphi+r+\lambda f(\varphi)\ \mathrm{mod}\  1,r+\lambda f(\varphi)),
\end{array}
\right.
$$
where $f$ is a smooth $1$-periodic function with mean value  zero and $\lambda$ is a coupling constant. The main object of interest is the set of 
quasi-periodic trajectories
for the map $\psi_{\lambda f}$ for a fixed irrational rotation number $\alpha$. One can show that such trajectories always exist. Moreover,  for some $a=a(\alpha) \in \R$, they correspond to  
minimizers for the Lagrangian action $A_{\lambda f,\alpha}$:
$$
A_{\lambda f,\alpha}((u_n)_{n \in \Z}):=\frac{1}{2}\sum_{n \in \Z} (u_{n+1}-u_n-a)^2 +\lambda \sum_{n \in \Z} F(u_n),
$$
where $F'=f$. 
However the behavior of such minimizers depends on the coupling constant $\lambda$. For small values of $|\lambda|$ the set of minimizers belongs to a 
smooth invariant 
curve, and under certain arithmetic conditions, the dynamics on this curve is conjugated to the rigid rotation by angle $\alpha$. On the other hand, for large values of $|\lambda|$ minimizers form  
disjoint 
Cantor-type sets, sometimes called \textit{cantori}. The theory describing properties of such invariant sets is often called the weak KAM theory. It is 
believed that one has a sharp transition in the parameter $\lambda$. Namely, for $\lambda < \lambda_{cr}(\alpha)$, there exists a smooth invariant curve (analytic if $f$ is analytic), for $\lambda = \lambda_{cr}(\alpha)$, an invariant curve still exists but it is not $C^{\infty}$ anymore (perhaps only 
$C^{1+\epsilon}$-smooth),
so that the conjugacy with the  rigid rotation is only topological, and the invariant measure is singular. Finally, for $\lambda > \lambda_{cr}(\alpha)$ the 
minimizers form
a Cantor-type set of zero Lebesgue measure. In fact it is even expected that  the Hausdorff dimension of such sets vanishes. It is also expected that
the dynamical properties of minimizing trajectories are very different before and after the transition. The trajectories belonging to smooth invariant 
curves are
elliptic with zero Lyapunov exponents, while trajectories belonging to cantori are expected to be hyperbolic. Although the above transition is
confirmed by many numerical studies, at present there are very few rigorous results in this direction, especially related to the critical case
$\lambda=\lambda_{cr}(\alpha)$.

Minimization of the Lagrangian action leads to the discrete Euler-Lagrange equation.  
In particular, any minimizing sequence $(\varphi_n)_{n\in \Z}$ must be related to a trajectory of the map $\psi_{\lambda f}$. The second differential of the action functional can be written as a quadratic 
form $(\mathcal{H}u,u)$,
where $\mathcal{H}$ is the 1D Schr\"odinger operator
$$
\mathcal{H}\colon
(u_n)_{n\in \Z} \mapsto (u_{n+1}+u_{n-1}+V_0(\varphi_n) u_n)_{n\in \Z}, 
$$
with $V_0:=-f'-2$. 

 Since a minimizing trajectory is quasi-periodic, the corresponding Schr\"odinger operator will be an 
operator with
a quasi-periodic potential. In fact, this construction leads to a one-parameter family of such potentials parametrized by the coupling constant $\lambda$.
It turns out that the potentials are smooth for $\lambda < \lambda_{cr}(\alpha)$ and discontinuous in the case $\lambda > \lambda_{cr}(\alpha)$. 

In this paper we propose to study the spectral properties of such families of Schr\"odinger operators.  We prove that in the KAM regime, when
there exists an analytic invariant curve, the Schr\"odinger operator has a component of absolutely continuous spectrum. We construct such a component 
near the edge of the spectrum. More precisely the following Main Theorem holds.\\

\textbf{Main Theorem.} Let $f$ be an analytic $1$-periodic function with  zero mean value. Suppose the standard-type map $\psi_{\lambda f}$ has an analytic invariant curve
homotopic to the base with a rotation number $\alpha$ of Brjuno type (see \eqref{Brjuno condtion} for a definition). Then the energy $E=0$ belongs to the spectrum of the Schr\"odinger operator $\mathcal{H}$, $E=0$ is the right 
edge of the spectrum: $\Sigma(\mathcal{H}) \cap (0, +\infty)= \emptyset$, and there exists $\varepsilon_0>0$ such that the spectral measures on $[-\varepsilon_0, 0]$
are absolutely continuous (and positive).\\

A more technical formulation of the theorem will be given in Section \ref{subs main res}. The proof is based on a dynamical argument. It is easy to see  
that a  Schr\"odinger  cocycle for the energy $E=0$ is conjugated to the dynamical (Jacobi) cocycle associated to $\psi_{\lambda f}$. This allows us to show that it is, in fact, reducible to a  
constant parabolic cocycle. Then, using the arguments developed by Avila \cite{A1}, we can conclude that the spectrum is absolutely continuous in a 
neighborhood of $E=0$.

This component of absolutely continuous spectrum comes from almost reducibility properties of Schr\"odinger  cocycles for energies near $E=0$. In \cite{Avila1}, Avila showed that almost reducibility is stable among analytic quasi-periodic  $\mathrm{SL}(2,\R)$-cocycles  with irrational frequency. In a similar vein, the component of absolutely continuous  spectrum  we obtain is also stable in the following sense: when the invariant curve has Diophantine rotation number $\alpha$, it persists under small analytic perturbations of the potential. Moreover, by  \cite{EFK}, this curve is also accumulated by other analytic invariant  curves with Diophantine rotation numbers close to $\alpha$. Our  Main Theorem guarantees stability in both senses, namely, under small analytic perturbations of the potential, and for these Diophantine rotation numbers,   the associated   Schr\"odinger operator  
also has a component of absolutely continuous spectrum.  

We view this theorem as a semi-global result. Namely we do not assume that the potential in the Schr\"odinger  operator is small. We only use the 
existence of an analytic invariant curve. We also conjecture that the critical value $\lambda_{cr}(\alpha)$ is a transition point.  
In other words, it is plausible that for all $\lambda>\lambda_{cr}(\alpha)$ the spectrum will be pure 
point. This would mean  that the dynamical transition from elliptic to hyperbolic behavior in the weak KAM theory reflects in a related transition 
in the spectral properties of the corresponding Schr\"odinger operators. It is an interesting problem to analyze the spectrum at the critical value $
\lambda=\lambda_{cr}(\alpha)$.

\subsection*{Acknowledgements} 
The authors would like to thank David Damanik, Svetlana Jitomirskaya, Rapha\"el Krikorian and Qi Zhou for useful discussions and encouragment on this project.

\section{Aubry-Mather theory and Schr\"odinger operators}\label{section prelimm}

\subsection{Conservative twist maps of the cylinder}

Let $C^\omega(\T,\R)$ be the set of one-periodic analytic functions on $\R$, and let $C_0^\omega(\T,\R)\subset C^\omega(\T,\R)$ be the subset of functions with zero average. 
For any   function  $f \in C_0^\omega(\T,\R)$, we let $\Psi=\Psi_f\colon \R^2 \to \R^2$, $(x,r) \mapsto   (x+r+f(x),r+f(x))$. It induces a map on the cylinder  $\A:=\T \times \R$:
$$
\psi=\psi_f\colon \left\{
\begin{array}{rcl}
\A &\to& \A,\\
(\varphi,r) &\mapsto& (\varphi+r+f(\varphi)\ \mathrm{mod}\  1,r+f(\varphi)).
\end{array}
\right.
$$
The map $\psi$ satisfies the twist property: for any $(\varphi,r) \in \A$, we have $\partial_r(\varphi+r+f(\varphi))=1>0$. 
For an arbitrary $f \in C_0^\omega(\T,\R)$, it is also natural to define a one-parameter family of twist maps  $(\psi_{\lambda f})_{\lambda \in \R}$, where $\lambda$ is called the \textit{coupling constant}. 

An important  case corresponds to the function $f_0\colon \varphi \mapsto \frac{1}{2\pi} \sin(2 \pi \varphi)$.  In this case, for any $\lambda \in \R$, the map $\psi_\lambda:=\psi_{\lambda f_0}$ is the \textit{Standard Map} with parameter $\lambda$. 

Given a point $(x,r)=(x_0,r_0) \in \R^2$, we let $((x_n,r_n))_{n \in \Z}$ be its orbit under $\Psi$:
\begin{equation}\label{eq orbit}
\left\{
\begin{array}{rcl}
x_{n+1} &=& x_n+r_{n+1},\\ 
r_{n+1} &=& r_n +f(x_n).
\end{array}
\right.
\end{equation}
Similarly, for $(\varphi,r)=(\varphi_0,r_0) \in \mathbb{A}$, we denote by $((\varphi_n,r_n))_{n \in \Z}$ its orbit under $\psi$; for $x_0\in \R$   such that $x_0\ \mathrm{mod}\  1=\varphi_0$, we have $(\varphi_n,r_n)=(x_n\ \mathrm{mod}\  1,r_n)$, for all $n \in \Z$. 

Let  $f \in C_0^\omega(\T,\R)$. The matrix of the differential of $\psi=\psi_f$ at $(\varphi,r)\in \A$ is  
$$
D\psi_{(\varphi,r)} =\begin{pmatrix}
1+f'(\varphi) & 1\\
f'(\varphi) & 1
\end{pmatrix} \in \mathrm{SL}(2,\R),
$$ 
hence $\psi$ is an analytic volume-preserving diffeomorphism with zero Calabi invariant:  $\psi\in \mathrm{Diff}^\omega_{m,0}(\A)$, where $dm=d\varphi dr$ is the Lebesgue measure. 

Let $F \in C^{\omega}_0(\T,\R)$ satisfy $F'=f$. For any $a \in \R$, we define a function  $h_{f,a}\colon \R^2 \to \R$ by the formula
$$
h_{f,a}(x_0,x_1):=\frac 12 (x_1-x_0-a)^2 + F(x_0), \qquad \forall\, (x_0,x_1) \in \R^2.
$$

Let $(\varphi_0,r_0) \in \A$, and let $(\varphi_1,r_1):=\psi(\varphi_0,r_0)$. Given $x_0 \in \R$  such that $x_0\ \mathrm{mod}\  1=\varphi_0$, we set $(x_1,r_1):=\Psi_f(x_0,r_0)$. The value $h_{f,a}(x_0,x_1)$ is independent of the choice of the lift $x_0 \in \R$ of $\varphi_0$, and thus, we may define
$$
h_{f,a}(\varphi_0,\varphi_1):=h_{f,a}(x_0,x_1). 
$$
In particular, given any two consecutive points $(\varphi_n,r_n)$, $(\varphi_{n+1},r_{n+1})$ in the orbit of $(\varphi,r)=(\varphi_0,r_0)\in \mathbb{A}$, the function $h_{f,a} (\varphi_n,\varphi_{n+1})$ is well-defined. 

Moreover, the function $h_{f,a}$  is \textit{generating} for $\psi=\psi_f$ in the following sense:  for $(\varphi_0,r_0) \in \A$ and $(\varphi_1,r_1):=\psi(\varphi_0,r_0)$, we have
\begin{equation*}
\left\{
\begin{array}{rcl}
\partial_1 h_{f,a}(\varphi_0,\varphi_1)&=&-(r_0-a),\\
\partial_2 h_{f,a}(\varphi_0,\varphi_1)&=& (r_1-a).
\end{array}
\right.
\end{equation*} 

\subsection{Action-minimizing Aubry Mather sets}

Let $f \in C_0^\omega(\T,\R)$, and let $F \in C^{\omega}_0(\T,\R)$ be the anti-derivative of $f$ with zero average.
Given $a \in \R$, 
we define the action of a sequence $u=(u_n)_{n \in \Z}\in \R^\Z$ as a formal sum:
$$
A_{f,a}(u):=\sum_{n \in \Z} h_{f,a}(u_n,u_{n+1})=\frac{1}{2}\sum_{n \in \Z} (u_{n+1}-u_n-a)^2 +\sum_{n \in \Z} F(u_n).
$$
The sequence $u=(u_n)_{n \in \Z}$ is called a \textit{minimizer of the action} $A_{f,a}$ if for any compact perturbation $\tilde u=(\tilde u_n)_{n \in \Z}$ of $(u_n)_{n \in \Z}$, the difference in action satisfies $A_{f,a}(\tilde u)-A_{f,a}(u)\geq 0$. Notice that the difference of actions is well defined although $A_{f,a}$ itself is just a formal series. 


Recall that minimizers of the action are associated to orbits of $\Psi_f$:
\begin{lemma}\label{lemma rappel}
Let $a \in \R$, and assume that the sequence $(x_n)_{n \in \Z}\in \R^\Z$ is a minimizer of the action $A_{f,a}$. 
We set $r_n:=x_{n}-x_{n-1}$, for all $n\in \Z$. Then, $((x_n,r_n))_{n \in \Z}$ is the orbit of $(x_0,r_0)$ under $\Psi_f$, and its projection $((\varphi_n,r_n))_{n \in \Z}$ on $\A$ is the orbit of $(\varphi_0,r_0)$ under $\psi_f$, with $\varphi_n:=x_n\ \mathrm{mod}\  1$, for all $n \in \Z$.  
\end{lemma}

\begin{proof}
Given any $u=(u_n)_{n \in \Z}\in \R^\Z$ and $\delta=(\delta_n)_{n \in \Z}\in \R^\Z$ satisfying $\delta_n=0$ for all but finitely many  integers $n \in \Z$, with $\|\delta\|:=\sqrt{\sum_k \delta_k^2}\ll 1$, we obtain
\begin{equation}\label{eqaction}
A_{f,a}(u+\delta)-A_{f,a}(u)=-\sum_{n \in \Z} (u_{n+1}-2u_n+u_{n-1}-f(u_n))\delta_n+O(\|\delta\|^2),
\end{equation}
where $(u+\delta)_n:=u_n+\delta_n$. Now, assume that the sequence $(x_n)_{n \in \Z}\in \R^\Z$ is a minimizer of the action $A_{f,a}$, and set $r_n:=x_{n}-x_{n-1}$, for all $n\in \Z$. Since $\delta$ can be taken arbitrarily small, we deduce that 
\begin{align*}
r_{n+1}&=x_{n+1}-x_{n}=x_{n}-x_{n-1}+f(x_n)=r_n+f(x_n),\\
x_{n+1}&=2 x_{n} - x_{n-1}+f(x_n)=x_n + r_{n+1}, 
\end{align*}  
for each $n\in \Z$. 
As a result,  \eqref{eq orbit} is satisfied, and $((x_n,r_n))_{n \in \Z}$ is the orbit of $(x_0,r_0)$ under $\Psi_f$.  Then, $((\varphi_n,r_n))_{n \in \Z}$ is the orbit of $(\varphi_0,r_0)$ under $\psi_f$, with $\varphi_n:=x_n\ \mathrm{mod}\  1$, for all $n \in \Z$.  
\end{proof}

\begin{remark}
The calculation above is just a derivation of the discrete Euler-Lagrange equation associated with the action given by $A_{f,a}$. 
\end{remark}

A $\psi_f$-invariant compact set $\mathcal{A} \subset \mathbb{A}$ is said to be $\psi_f$\textit{-ordered} if it projects injectively on $\T$, and the restriction $\psi_f|_{\mathcal{A}}$  preserves the natural order given by the projection. A classical result of Aubry-Mather theory   \cite{Aubry,AubryDaeron,Mather} states that for each irrational number $\alpha \in \R \setminus\Q$, there exist $a=a(\alpha) \in \R$ and a minimal ordered set $\mathcal{AM}_{f,\alpha}\subset \A$ with rotation number $\alpha$. It is comprised of  orbits $((\varphi_n,r_n))_{n \in \Z}$ under $\psi_f$ associated to sequences $(x_n)_{n \in \Z}$ which minimize the action $A_{f,a}$.  By some slight abuse of notation, we denote $A_{f,\alpha}:=A_{f,a(\alpha)}$ in the following. In particular, the rotation number $\alpha$ of $\mathcal{AM}_{f,\alpha}$ is  the rotation number  of any lifted orbit: for any $(\varphi_0,r_0)=(x_0\ \mathrm{mod}\  1,r_0)  \in \mathcal{AM}_{f,\alpha}$, we have 
$$
\alpha=\lim_{n \to+\infty} \frac{x_n- x_0}{n}=\lim_{n \to+\infty} \frac 1n \sum_{k=1}^n r_k. 
$$  
The set  $\mathcal{AM}_{f,\alpha}$ is  called the \textit{minimizing Aubry-Mather set} for the action $A_{f,\alpha}$. 
%
\begin{theorem}[\cite{Aubry,AubryDaeron,Mather,Birkhoff,Fathi,Herman}]\label{theo aurby mather}
For any $\alpha \in \R \setminus\Q$,   the associated minimizing Aubry-Mather set  $\mathcal{AM}_{f,\alpha}$  is either an invariant  graph $\Gamma_\gamma$ for some Lipschitz function $\gamma \colon \T \to \R$, or it projects one-to-one to a nowhere-dense Cantor set of $\T$. 
Moreover, if $\Gamma$ is an  invariant  curve  for $\psi_f$ homotopic to the base 
with irrational  rotation number $\alpha$,  it is a minimizing  Aubry-Mather set: we have $\Gamma=\mathcal{AM}_{f,\alpha}$. 
\end{theorem}

Suppose that $\psi_f$ leaves invariant the   graph $\Gamma_\gamma=\{\Gamma_\gamma(\varphi):=(\varphi,\gamma(\varphi)): \varphi \in \T\}$\footnote{We shall use the same notation $\Gamma_\gamma$ for the natural map from $\T$ to $\A$ defined by the graph $\Gamma_\gamma$.}  with rotation number $\alpha \  \mathrm{mod}\ 1$ for some $\alpha \in \R\setminus\Q$. Then, the composition $\pi_1 \circ \psi_f \circ \Gamma_\gamma$ yields a circle homeomorphism $g=g_\gamma\colon \T \to \T$, where  $\pi_1 \colon \mathbb{A} \to \T$ is the projection on the first coordinate: 
$$
g(\varphi) :=\varphi+\gamma(\varphi)+f(\varphi)\ \mathrm{mod}\ 1,\qquad \forall\, \varphi \in \T.
$$ 
The map  $G=G_\gamma \colon \R \ni x \mapsto x+\gamma(x)+f(x)$ is a lift of $g$, 
and it has rotation number $\alpha\in \R\setminus\Q$. Since $g$ has rotation number $\alpha\  \mathrm{mod}\ 1$,  it is uniquely ergodic. We denote by $\nu$ its unique invariant probability measure. For any orbit $((\varphi_n,r_n))_{n \in \Z} \subset \mathcal{AM}_{f,\alpha}$, we have $r_k=\gamma(\varphi_k)=\gamma( g^k(\varphi_0))$, for all $k \in \Z$, hence 
$$
\alpha=\lim_{n \to+\infty} \frac 1n \sum_{k=1}^n r_k=\lim_{n \to+\infty} \frac 1n \sum_{k=1}^n \gamma( g^k(\varphi_0))=\int_{\T} \gamma(\varphi)\, d\nu(\varphi).
$$


\subsection{Dynamically defined quasi-periodic Schr\"odinger operators}\label{subsectin finfd}

Let $f \in C_0^\omega(\T,\R)$.  Fix $\alpha \in \R\setminus\Q$ and let us consider the  minimizing Aubry-Mather set $\mathcal{AM}_{f,\alpha}$. 
In the following, we choose a phase $\varphi_0 \in \T$ such that  $(\varphi_0,r_0) \in \mathcal{AM}_{f,\alpha}$ for some $r_0 \in \R$. Given $x_0 \in \R$ such that $\varphi_0=x_0 \ \mathrm{mod}\  1$, we denote by $((x_n,r_n))_{n \in \Z}$ the orbit of $(x_0,r_0)$ under $\Psi_f$ and let $((\varphi_n,r_n))_{n \in \Z}$ be the corresponding $\psi_f$-orbit. We  also denote by $(\cdot,\cdot)$ the standard inner product on $\ell^2(\Z)$,  and for $u \in \ell^2(\Z)$, we set $\|u\|:=\sqrt{(u,u)}$.   

In the proof of Lemma \ref{lemma rappel}, we have computed the first order term in the difference of actions between a sequence and a compact perturbation of it. It turns out that the second order term in this difference is given by some quadratic form associated to a Schr\"odinger operator, as shown by the next lemma. 

\begin{lemma}
	For any sequence $\delta=(\delta_n)_{n \in \Z}\in \R^\Z$ satisfying $\delta_n=0$ for all but finitely many integers $n \in \Z$ and such that $\|\delta\|\ll 1$, we have
	\begin{equation}\label{new eq action}
	A_{f,\alpha}((x_n+\delta_n)_{n\in \Z})-A_{f,\alpha}((x_n)_{n\in \Z})=-\frac 12(\mathcal{H}_{f,\alpha,\varphi_0}\delta,\delta)+O(\|\delta\|^3),
	\end{equation}
	where $\mathcal{H}_{f,\alpha,\varphi_0}$  is the dynamically defined \textit{Schr\"odinger operator}  associated to the analytic function $V_0:=-f'-2$ and the  $\psi_f$-ordered sequence $(\varphi_n)_{n \in \Z}$:
	$$
	\mathcal{H}_{f,\alpha,\varphi_0}\colon \left\{
	\begin{array}{rcl}
	\ell^2(\Z) &\to &\ell^2(\Z),\\
	(u_n)_{n\in \Z} &\mapsto& (u_{n+1}+u_{n-1}+V_0(\varphi_n) u_n)_{n\in \Z}.
	\end{array}
	\right.
	$$
\end{lemma}

\begin{proof}
Since $((x_n,r_n))_{n \in \Z}$ is a $\Psi_f$-orbit, we have $x_{n+1}-2x_n+x_{n-1}-f(x_n)=0$, for all $n \in \Z$. Equivalently, $(x_n)_{n \in \Z}$ is a critical point for the action $A_{f,\alpha}$, and the first order term in the difference \eqref{eqaction} vanishes for $(x_n)_{n \in \Z}$ 
in place of $(u_n)_{n \in \Z}$. 
Given any compact perturbation $(x_n+\delta_n)_{n \in \Z}$ of $(x_n)_{n \in \Z}$ for some sequence $\delta=(\delta_n)_{n \in \Z}\in \R^\Z$ with $\|\delta\|\ll 1$, the Taylor expansion of the difference in action is
\begin{align*}
	A_{f,\alpha}((x_n+\delta_n)_{n\in \Z})-A_{f,\alpha}((x_n)_{n\in \Z})&=\sum_{n \in \Z}\frac 12(\delta_{n+1}-\delta_n)^2 +f'(x_n) \frac{\delta_n^2}{2}+O(\|\delta\|^3)\\
	&=-\frac 12(\mathcal{H}_{f,\alpha,\varphi_0}\delta,\delta)+O(\|\delta\|^3).
\end{align*}
\end{proof}
 
In particular, for any $\psi_f$-ordered sequence $(\varphi_n)_{n \in \Z}\subset \mathcal{AM}_{f,\alpha}$, 
  $\mathcal{H}=\mathcal{H}_{f,\alpha,\varphi_0}$ is a quasi-periodic Schr\"odinger operator whose potential is related to the dynamics of $\psi_f|_{\mathcal{AM}_{f,\alpha}}$.  We denote by $\Sigma=\Sigma(\mathcal{H})$ its \textit{spectrum}; recall that it is the set of energies $E$ such that the operator $\mathcal{H}-E=\mathcal{H}-E \cdot \mathrm{I}$ does not have a bounded inverse in $\ell^2(\Z)$. 
For any $u \in \ell^2(\Z)$, we denote by $\mu_\mathcal{H}^u$ the spectral measure  of $\mathcal{H}$ associated to $u$. It is defined by the following formula:
$$
((\mathcal{H}-E)^{-1}u,u)=\int_{\R} \frac{1}{E'-E}d\mu_\mathcal{H}^u(E'),
$$
for any energy $E$ in the resolvent set  $\C\setminus\Sigma$. The union of the supports of all spectral measures is equal to $\Sigma$.    We refer for instance to \cite{Damanik} for more details on dynamically defined Schr\"odinger operators. 


\begin{lemma}\label{lemma spectre}
For any $u \in \ell^2(\Z)$, we have
\begin{equation}\label{new eq action2}
(\mathcal{H}u,u)\leq 0.
\end{equation}
Equivalently,
\begin{equation}\label{new eq action2 bibibis}
\quad \Sigma(\mathcal{H})\subset (-\infty,0].
\end{equation}
\end{lemma}

\begin{proof}
	We first show \eqref{new eq action2}. 
Let $u \in \ell^2(\Z)$.  
For any integer $n_0\geq 1$, we define the compact sequence $u^{n_0}=(u^{n_0}_n)_{n \in \Z}$, where $u^{n_0}_n:=u_n$ if $n \in [-n_0,n_0]$, and $u_n^{n_0}:=0$ otherwise.  
As $(x_n)_{n \in \Z}$ is a minimizer of $A_{f,\alpha}$, for any integer $k \geq 1$, the following expression is always nonnegative:
\begin{equation*}
	A_{f,\alpha}\Big(\big(x_n+\frac 1k u_n^{n_0}\big)_{n\in \Z}\Big)-A_{f,\alpha}((x_n)_{n\in \Z})=-\frac {1}{2k^2}\big(\mathcal{H}u^{n_0},u^{n_0}\big)+O\Big(\frac{1}{k^3}\|u^{n_0}\|^3\Big).
\end{equation*}
For $k \gg 1$  large, and as the error term in the above Taylor expansion scales like $\frac{1}{k^3}$,  we deduce that $\big(\mathcal{H}u^{n_0},u^{n_0}\big) \leq 0$. Letting $n_0 \to +\infty$, we conclude that   $(\mathcal{H} u,u) \leq 0$.  




We now prove the second statement \eqref{new eq action2 bibibis} about the spectrum $\Sigma=\Sigma(\mathcal{H})$. It amounts to showing that for any $E>0$, the operator $\mathcal{H}-E$ is invertible and that its inverse is bounded. Fix $E>0$. By \eqref{new eq action2}, for any sequence $u \in \ell^2(\Z)$, we have 
\begin{equation}\label{inequ normes}
\|(\mathcal{H}-E)u\|^2=\|\mathcal{H}u\|^2-2E(\mathcal{H}u,u)+E^2\|u\|^2\geq E^2\|u\|^2.
\end{equation}
Therefore, if $(\mathcal{H}-E)u=0$, then  $\|u\|=0$, and thus, $\mathcal{H}-E$ is injective. Moreover, \eqref{inequ normes} also implies that the inverse of $\mathcal{H}-E$ is bounded, since 
\begin{equation*}\label{inequ normes deux}
\sup_{u \in \ell^2(\Z)}\frac{\|(\mathcal{H}-E)^{-1}u\|}{\|u\|}=\left(\inf_{u \in \ell^2(\Z)}\frac{\|(\mathcal{H}-E)u\|}{\|u\|}\right)^{-1}\leq \frac{1}{E}. 
\end{equation*}
Besides, $\mathcal{H}-E$ is automatically surjective. Indeed, it holds $\mathrm{Im}(\mathcal{H}-E)^{\perp}=\mathrm{Ker}(\mathcal{H}-E)^*=\mathrm{Ker}(\mathcal{H}-E)=\{0\}$, hence $\overline{\mathrm{Im}(\mathcal{H}-E)}=\ell^2(\Z)$. By \eqref{inequ normes} and Cauchy's convergence test, we deduce that $\mathrm{Im}(\mathcal{H}-E)=\ell^2(\Z)$. We conclude that $\Sigma\subset(-\infty,0]$. 

Conversely, let us show  that $\Sigma\subset(-\infty,0]$ implies \eqref{new eq action2}. Indeed, the spectrum $\Sigma$ is the union of the supports of the spectral measures $\{\mu_\mathcal{H}^u\}_{u \in \ell^2(\Z)}$, hence for any $u \in \ell^2(\Z)$, the support of $\mu_\mathcal{H}^u$ is contained in $(-\infty,0]$, and
$$
(\mathcal{H}u,u)=\int_\R E\, d\mu_\mathcal{H}^u(E)=\int_{-\infty}^0 E\, d\mu_\mathcal{H}^u(E)\leq 0.
$$
\end{proof}

\section{Main Theorem}\label{subs main res}

Let $f \in C_0^\omega(\T,\R)$, and let $\psi_f\colon \mathbb{A} \to \mathbb{A}$ be the associated twist map. 
Assume that $\psi_f$ leaves invariant an analytic curve $\Gamma$ 
which is the graph 
$\Gamma=\Gamma_\gamma:=\{(\varphi,\gamma(\varphi)): \varphi \in \T\}$ of some function $\gamma \in C^\omega(\T,\R)$. We denote by  
$$
g\colon \varphi \mapsto \varphi+\gamma(\varphi)+f(\varphi)\ \mathrm{mod}\ 1
$$ 
the  analytic diffeomorphism of $\T$ induced by $\psi_f|_{\Gamma_\gamma}$, and assume that the rotation number $\alpha$ of $g$ satisfies the \textit{Brjuno condition}, i.e.,
\begin{equation}\label{Brjuno condtion}
	\mathfrak{B}(\alpha):=\sum_{k=0}^{+\infty} \frac{1}{q_{k}} \log q_{k+1} < +\infty,
\end{equation}
where $\big(\frac{p_k}{q_k}\big)_{k \geq 0}$ denotes the sequence of 
convergents for $\alpha$ (corresponding to the continued fraction algorithm). 
In particular, we can use the results of  \cite{Avila1,AvilaJito,LYZZ}, 
where it is assumed that $\beta(\alpha):=\limsup_{k} \frac{1}{q_{k}} \log q_{k+1}=0$.  


As above, given $\varphi_0 \in \T$, we consider the  Schr\"odinger operator 
$$
\mathcal{H} \colon 
(u_n)_{n\in \Z} \mapsto (u_{n+1}+u_{n-1}+V_0(g^n(\varphi_0)) u_n)_{n\in \Z},
$$ 
with $V_0:=-f'-2$, 
and denote by $\Sigma(\mathcal{H})$ its spectrum. 

Our main result is:

\begin{thmh}\label{theo alpha first}
	Assume that the map $\psi_f$  leaves invariant an analytic  curve $\Gamma$ with Brjuno rotation number $\alpha$, and let $\mathcal{H}$ be the associated Schr\"odinger operator.  
	Then there exists $\varepsilon_0 >0$ such that the following properties hold:
	\begin{enumerate}
		\item  the energy $E=0$ is the right edge of the spectrum: $0=\max  \Sigma(\mathcal{H})$;
		\item  the spectral measures of $\mathcal{H}$ restricted to $[-\varepsilon_0,0]$ are absolutely continuous;
		\item there exists $\kappa>0$ such that 
		$| (E-\varepsilon,E+\varepsilon) \cap \Sigma(\mathcal{H})| > \kappa\varepsilon$, for all energy $E\in \Sigma(\mathcal{H})\cap (-\varepsilon_0,0)$, and for all $0<\varepsilon<|E|$.
	\end{enumerate}
\end{thmh}

The proof is based on the reducibility of the \textit{Schr\"odinger cocycle} associated to  $\mathcal{H}$ for the energy $E=0$. As we shall explain, this can be seen in two ways:
\begin{enumerate}
	\item in restriction to $\Gamma$, the \textit{Jacobi (differential) cocycle} of the twist map $\psi_f$ is conjugate to some quasi-periodic Schr\"odinger cocycle (see Subsections \ref{subs jacobi}-\ref{schr coc}). Besides, in Subsection \ref{subs parabbbo}, we  use a vector field tangent  to $\Gamma$ to reduce these cocycles to a constant cocycle associated to some parabolic matrix. 
	\item the energy $E=0$ is in the pure point spectrum of some \textit{dual Schr\"odinger operator} (see Subsection \ref{peure poinnt}). 
	Actually, thanks to the existence of the invariant curve $\Gamma$, we construct an explicit eigenvector   whose coefficients decay exponentially  fast. 
\end{enumerate} 
By Avila's results, this implies that for small energies, the corresponding  Schr\"odinger cocycles are \textit{almost reducible}, i.e., they can be conjugated uniformly in some strip to a cocycle which is arbitrarly close to a constant. Finally, we follow the proof given by Avila in \cite{A1} to  show the existence of a  component of absolutely continuous spectrum near the energy $E=0$.

\begin{remark}
	$\bullet$ Such properties are typical of the regime of small analytic potentials (see \cite{A1,AvilaJito,Eliasson} for instance). Our result  replaces the usual smallness assumption with the geometric assumption on the existence of an analytic invariant curve. 
	
	$\bullet$  Let us also recall that for maps $\psi_f$ as above, the existence of analytic  invariant curves with a given Brjuno rotation number  is guaranteed by the main result of \cite{Gentile}, provided that the analytic norm of $f$ is sufficiently small. More precisely, by Theorem 1.1 in \cite{Gentile}, for any $f_0 \in C_0^\omega(\T,\R)$ and for any $\alpha\in \R$ satisfying the Brjuno condition $\mathfrak{B}(\alpha)<+\infty$, there exists $\lambda_0>0$ such that for $|\lambda|<\lambda_0$, the map $\psi_{\lambda f_0}$ admits an analytic invariant curve  with rotation number $\alpha$. 
\end{remark}

\section{Invariant curves \& almost reducibility of Schr\"odinger cocycles}

\subsection{Invariant curves \& Jacobi (differential) cocycle}\label{subs jacobi}


Let $f \in C_0^\omega(\T,\R)$ be an analytic function with zero average. 
For the map $\psi=\psi_f$ one can define in a usual way the Jacobi cocycle $(\psi,D \psi)\colon \A \times \C^2 \to \A \times \C^2$, namely
$$
(\psi,D \psi)((\varphi,r),v):=(\psi(\varphi,r),D \psi_{(\varphi,r)}\cdot v), \quad \forall\, ((\varphi,r),v) \in \A \times \C^2.
$$

  Assume that  $\Gamma\subset \mathbb{A}$ is an invariant curve for $\psi$ homotopic to the base.  
 As recalled in Theorem \ref{theo aurby mather}, by Birkhoff Theorem, 
 $\Gamma$ is 
 the graph $\Gamma_\gamma=\{\Gamma_\gamma(\varphi)=(\varphi,\gamma(\varphi)):\varphi \in \T\}$ of some Lipschitz function $\gamma \colon \T  \to \R$.  
 As above, we let
 $
 g=g_\gamma\colon \varphi \mapsto \varphi+\gamma(\varphi)+f(\varphi) \ \mathrm{mod}\  1
 $ 
 be the circle map obtained by projecting $\psi|_{\Gamma_\gamma}$ on the first coordinate. For any $\varphi \in \T$, we have 
 \begin{equation}\label{equation graphe}
 \psi (\Gamma_\gamma(\varphi))=\psi (\varphi,\gamma(\varphi))=(g(\varphi),\gamma(\varphi)+f(\varphi))=(g(\varphi),\gamma \circ g(\varphi))\in \Gamma_\gamma. 
 \end{equation}
 Clearly, $g$ is a homeomorphism of $\T$; moreover, $g$ and $\gamma$ have the same regularity. In particular, in the case we consider, $\gamma$ and $g$ are analytic. 
 
As the curve $\Gamma$ is invariant under $\psi_f$, the restriction of the Jacobi cocycle to $\Gamma$ reduces to the  derivative cocycle $(g,D\psi)\colon \T \times \C^2 \to \T \times \C^2$: 
\begin{equation}\label{def cocycle}
(g,D\psi)(\varphi,v):=(g(\varphi),D\psi(\varphi)\cdot v),\quad \forall\, (\varphi,v) \in \T \times \C^2,
\end{equation}
where $D\psi(\varphi):=D\psi_{(\varphi,\gamma(\varphi))}$.  
Let us  denote by $v_0$ a vector field tangent to the invariant curve: 
$$
v_0\colon \varphi\mapsto\begin{pmatrix} 1\\
\gamma'(\varphi)
\end{pmatrix}\in T_{(\varphi,\gamma(\varphi))}\Gamma.
$$
It is easy to see that $v_0$ is an invariant  section for the derivative cocycle in the directional sense.
\begin{lemma}\label{lemma un}
The action of the cocycle $(g,D\psi)$ on the vector field $v_0$ is given by
\begin{equation}\label{inv section}
D\psi(\varphi) \cdot v_0(\varphi)=g'(\varphi)\cdot v_0(g(\varphi)),\quad \forall\, \varphi \in \T,
\end{equation}
with $g'(\varphi)=1+\gamma'(\varphi)+f'(\varphi)=(1-\gamma'\circ g(\varphi))^{-1}$.

Moreover, there exists an analytic conjugacy map $Z_1\in C^\omega (\T, \mathrm{SL}(2,\R))$  such that 
\begin{equation}\label{first conj}
(Z_1\circ g(\varphi))^{-1} D\psi(\varphi) Z_1(\varphi)=\begin{pmatrix}
g'(\varphi) & 1\\
0 & g'(\varphi)^{-1}
\end{pmatrix},\quad \forall\, \varphi \in \T.
\end{equation}
\end{lemma}

\begin{proof}
As we have seen in \eqref{equation graphe}, the fact that $\Gamma$ is invariant means that $f$ can be written as a coboundary, i.e., it satisfies the cohomological equation
\begin{equation}\label{cohomological eq}
f(\varphi)=\gamma\circ g(\varphi)-\gamma(\varphi),\quad \forall\, \varphi \in \T.
\end{equation}
Differentiating with respect to $\varphi$ in \eqref{cohomological eq}, we obtain: 
\begin{equation}\label{gamma' inv}
 \gamma'\circ g(\varphi)\cdot g'(\varphi) =\gamma'(\varphi)+f'(\varphi),
\end{equation}
with $g'(\varphi)=1+\gamma'(\varphi)+f'(\varphi)$. It also follows that $g'(\varphi)\cdot (1-\gamma'\circ g(\varphi))=1$. 

By  \eqref{gamma' inv}, we deduce that
$$
\begin{pmatrix}
1+f'(\varphi) & 1\\
f'(\varphi) & 1
\end{pmatrix}\cdot \begin{pmatrix} 1\\
\gamma'(\varphi)
\end{pmatrix}=g'(\varphi)\begin{pmatrix} 1\\
\gamma'\circ g(\varphi)
\end{pmatrix},
$$
i.e., $D\psi(\varphi) \cdot v_0(\varphi)=g'(\varphi) \cdot v_0(g(\varphi))$.  

Now, let us set
$$
Z_1(\varphi):=\begin{pmatrix}
1 & 0\\
\gamma'(\varphi) & 1
\end{pmatrix},\quad \forall\, \varphi \in \T.
$$ 
Then, by \eqref{inv section}, and since $ (g'(\varphi))^{-1}=1-\gamma'\circ g(\varphi)$, we see that the map $Z_1$ conjugates $(g,D\psi)$ to the 
upper-triangular cocycle as in  \eqref{first conj}. 
\end{proof}

\subsection{Schr\"odinger cocycle associated to an invariant curve}\label{schr coc}

Let $f$, $\psi=\psi_f$, $\Gamma$ and $g$ be as in the previous subsection. As we have seen above, one can define a natural family of dynamically generated Schr\"odinger operators, the phase $\varphi_0 \in \T$ being a parameter:
$$
\mathcal{H}_{f,\alpha,\varphi_0}\colon 
(u_n)_{n\in \Z} \mapsto (u_{n+1}+u_{n-1}+V_0(g^n(\varphi_0)) u_n)_{n\in \Z},
$$ 
where $V_0:=-f'-2$. For an arbitrary energy $E \in \R$, one can define a   Schr\"odinger cocycle  $(g,S_E^{V_0})$, with
$$
S_E^{V_0}(\varphi):=\begin{pmatrix}
E-V_0(\varphi) & -1\\
1 & 0
\end{pmatrix},\quad \forall\, \varphi \in \T.
$$
It acts on $\T \times \C^2$ by the following formula:
\begin{equation*}
	(g,S_E^{V_0})(\varphi,v):=(g(\varphi),S_E^{V_0}(\varphi)\cdot v),\quad \forall\, (\varphi,v) \in \T \times \C^2. 
\end{equation*}  
It turns out that the Schr\"odinger cocycle and the derivative cocycle are conjugate to each other for the energy $E=0$: for any $\varphi \in \T$, we have
\begin{equation}\label{identite matrice schrodd}
\begin{pmatrix}
1 & 0\\
1 & -1
\end{pmatrix}\begin{pmatrix}
1+f'(\varphi) & 1\\
f'(\varphi) & 1
\end{pmatrix}
\begin{pmatrix}
1 & 0\\
1 & -1
\end{pmatrix}=
\begin{pmatrix}
2+f'(\varphi) & -1\\
1 & 0
\end{pmatrix}.
\end{equation}

We next show that both cocycles $(g,S_0^{V_0})$ and $(g,D\psi)$ are conjugate to a constant parabolic cocycle $(g,B_0)$, for some matrix $B_0:=\begin{pmatrix}
1 & \nu_0\\
0 & 1
\end{pmatrix}$. This conjugation holds only along the invariant curve. To proceed, we need some facts from the Herman-Yoccoz theory.  

Assume that 
the rotation number $\alpha$ of $\Gamma$ satisfies the Brjuno condition 
$
\mathfrak{B}(\alpha):=\sum_{k=0}^{+\infty} \frac{1}{q_{k}} \log q_{k+1} < +\infty
$,
where $\big(\frac{p_k}{q_k}\big)_{k \geq 0}$ denotes the sequence of 
continued fraction convergents for $\alpha$. 
We shall also assume that $\Gamma=\Gamma_\gamma$ is  the graph  of some analytic function $\gamma\in C^\omega( \T, \R)$. 
It follows that the circle diffeomorphism $
g=g_\gamma$ 
    is also analytic. Hence, by the theorem of Yoccoz  (see  \cite{Yoccoz} and also \cite{EKMY,ElFaKr} for a reference), $g$ is analytically conjugate to the rigid rotation $r_\alpha$ by angle $\alpha$. Namely, there exists an analytic circle diffeomorphism $\phi =\phi_\gamma\in C^\omega(\T,\T)$ such that
\begin{equation}\label{def conjugacy phi}
\phi^{-1} \circ g  \circ \phi(\varphi)=\varphi+\alpha \ \mathrm{mod}\ 1,\quad \forall\, \varphi \in \T. 
\end{equation}

The following lemma says that both $f$ and $\gamma$ can be expressed in terms of the conjugacy map $\phi$. 

\begin{lemma}\label{lemma conjugajud}
	It holds
\begin{equation}\label{f vareps}
f =\gamma \circ g - \gamma=\gamma \circ \phi r_\alpha \phi^{-1}- \gamma.
\end{equation} 
Moreover, we have 
\begin{equation}\label{deq gammma}
\gamma=\mathrm{I}-g^{-1}=\mathrm{I}-\phi r_\alpha^{-1} \phi^{-1},
\end{equation}
where $\mathrm{I}$ denotes the identity map, 
which implies that
\begin{equation}\label{deq fgammama}
f=g-2 \mathrm{I}+g^{-1}=\phi r_\alpha \phi^{-1} - 2 \mathrm{I} + \phi r_\alpha^{-1} \phi^{-1}.
\end{equation}
\end{lemma}

\begin{proof}
The identity in \eqref{f vareps} follows directly  from  \eqref{cohomological eq}. 
As $g=g_\gamma$ is the circle diffeomorphism induced by $\psi_f|_{\Gamma_\gamma}$, for all $\varphi \in \T$, we have $g(\varphi)=\varphi+\gamma(g(\varphi))$, which implies $\gamma=\mathrm{I}-g^{-1}$, and hence, \eqref{deq gammma}. Finally,  \eqref{deq fgammama} follows from \eqref{f vareps} and \eqref{deq gammma}. 
\end{proof}

Conversely, given an irrational frequency $\tilde \alpha$ and an analytic circle diffeomorphism $\tilde \phi$, one can produce an analytic function $\tilde f$ with zero average such that the associated twist map $\psi_{\tilde f}$ has an analytic invariant  curve with rotation number $\tilde \alpha$. Moreover, $\tilde \phi$ conjugates the corresponding circle diffeomorphism $\tilde g$ to the rigid rotation $r_{\tilde \alpha}$. 
\begin{lemma}\label{prop beta}
	Let $\tilde \alpha\in \R\setminus\Q$ 
	and let $\tilde \phi \in C^\omega(\T,\T)$  be an analytic circle diffeomorphism.  We define two   functions $
	\tilde \gamma=\tilde \gamma_{\tilde \alpha,\tilde \phi}\in C^\omega(\T,\R)$ and $\tilde f=\tilde f_{\tilde \alpha,\tilde \phi} 
	\in C_0^\omega(\T,\R)$:
	\begin{align*}
	\tilde \gamma&:=\mathrm{I}-\tilde \phi r_{\tilde \alpha}^{-1}\tilde  \phi^{-1},\\
	\tilde f&:=\tilde \gamma\circ\tilde  \phi r_{\tilde \alpha}\tilde \phi^{-1}-\tilde \gamma=\tilde \phi r_{\tilde \alpha}\tilde   \phi^{-1}-2\mathrm{I}+\tilde \phi r_{\tilde \alpha}^{-1} \tilde \phi^{-1}.
	\end{align*}
	Then, 
	the analytic graph $\Gamma_{\tilde \gamma}:=\{\Gamma_{\tilde \gamma}(\varphi)=(\varphi,\tilde \gamma(\varphi)): \varphi \in \T\}$ is invariant under $ \psi_{\tilde f}$ with a  rotation number $\tilde \alpha$, and $\tilde \phi^{-1} \circ \tilde g \circ \phi= r_{\tilde \alpha}$. 
\end{lemma}

Let $f,\psi,\Gamma=\Gamma_\gamma,g=g_\gamma$ and $\phi=\phi_\gamma$ be as previously. For any map $A \in C^\omega(\T,\mathrm{SL}(2,\R))$, we denote by $(\alpha,A)$ the associated cocycle over the rigid rotation by angle $\alpha$. It acts on 
 $\T \times \C^2$ as follows:
\begin{equation*}
	(\alpha,A)(\varphi,v):=(\varphi+\alpha,A(\varphi)\cdot v),\quad \forall\, (\varphi,v) \in \T \times \C^2. 
\end{equation*}  
After conjugation by $\phi$, the derivative cocycle and the Schr\"odinger  cocycle yield  cocycles $(\alpha,D\psi\circ \phi)$ and $(\alpha,S_E^{V})$, $E \in \R$, where $V:=V_0 \circ \phi=-f' \circ \phi-2\in C^\omega(\T,\R)$. 
Besides, according to \eqref{identite matrice schrodd}, these cocycles are conjugated by the matrix $M :=\begin{pmatrix}
1 & 0 \\
1 & -1
\end{pmatrix}=M^{-1} \in \mathrm{GL}(2,\R)$ for the energy $E=0$:
\begin{equation}\label{link schr}
M^{-1} \cdot D\psi(\phi(\varphi)) \cdot M=S_0^V(\varphi),\quad \forall\, \varphi\in \T.
\end{equation}

The Schr\"odinger cocycles $\{(\alpha,S_E^{V})\}_{E \in \R}$ are associated to the family  of Schr\"odinger operators $\{H_{V,\alpha,\varphi_0}\}_{\varphi_0 \in \T}$  over  the dynamics of the rigid rotation $r_\alpha$: 
$$
H_{V,\alpha,\varphi_0}\colon 
(u_n)_{n\in \Z} \mapsto (u_{n+1}+u_{n-1}+V(\varphi_0+n\alpha) u_n)_{n\in \Z}.
$$
The spectrum $\Sigma(H_{V,\alpha,\varphi_0})$ does not depend on the phase $\varphi_0$. Recall that  $\Gamma=\mathcal{AM}_{f,\alpha}$, and note that $H_{V,\alpha,\phi^{-1}(\varphi_0)}
=\mathcal{H}_{f,\alpha,\varphi_0}$. In particular,  $\Sigma(H_{V,\alpha,\varphi_0})=\Sigma(\mathcal{H}_{f,\alpha,\varphi_0})$.

\subsection{Parabolic reducibility of the Schr\"odinger cocycle}\label{subs parabbbo}

We keep the notations of the previous subsection.
As a consequence of Lemma \ref{lemma un}, we obtain:
\begin{corollary}
Let $Z_2:=M \cdot Z_1\circ \phi\in C^\omega(\T,\mathrm{GL}(2,\R))$. Then, we have
\begin{equation}\label{first conj bbis}
Z_2(\varphi+\alpha)^{-1} S_0^{V} (\varphi) Z_2(\varphi)=\begin{pmatrix}
\kappa(\varphi) & 1\\
0 & \kappa(\varphi)^{-1}
\end{pmatrix},\quad \forall\, \varphi \in \T,
\end{equation}
with $\kappa(\varphi):=\frac{\phi'(\varphi+\alpha)}{\phi'(\varphi)}$. 
\end{corollary}

\begin{proof}
Recall that $\phi=\phi_\gamma$ satisfies $g \circ \phi(\varphi)=\phi(\varphi+\alpha)$, for all $\varphi \in \T$. Therefore, $g'\circ \phi(\varphi)=\frac{\phi'(\varphi+\alpha)}{\phi'(\varphi)}=:\kappa(\varphi)$. 
Set $\widetilde{Z}_1:=Z_1 \circ \phi$. By \eqref{first conj}, 
we deduce that 
\begin{equation*}
\widetilde{Z}_1(\varphi+\alpha)^{-1} D\psi(\phi(\varphi)) \widetilde{Z}_1(\varphi)=\begin{pmatrix}
\kappa(\varphi) & 1\\
0 & \kappa(\varphi)^{-1}
\end{pmatrix},\quad \forall\, \varphi \in \T.
\end{equation*}
Set $Z_2:=M \cdot \widetilde{Z}_1=M \cdot Z_1 \circ \phi$. By  \eqref{link schr}, 
we thus conclude that
\begin{equation*}
Z_2(\varphi+\alpha)^{-1} S_0^{V} (\varphi) Z_2(\varphi)=\begin{pmatrix}
\kappa(\varphi) & 1\\
0 & \kappa(\varphi)^{-1}
\end{pmatrix},\quad \forall\, \varphi \in \T.
\end{equation*} 
\end{proof}

 As a consequence of the above result, we show that for the energy $E=0$, the Schr\"odinger cocycle $(\alpha,S_0^{V})$ can be reduced to a parabolic cocycle. 

\begin{prop}\label{prop reductibilite parab}
There exist a negative number $\nu_0<0$ and an analytic conjugacy $Z \in C^\omega( \T, \mathrm{SL}(2,\R))$ homotopic to the identity such that 
$$
Z(\varphi+\alpha)^{-1}S_0^{V}(\varphi) Z(\varphi)=B_0:=\begin{pmatrix}
1 & \nu_0\\
0 & 1
\end{pmatrix},\quad \forall\, \varphi\in \T.
$$
\end{prop}

\begin{proof}
For any $\varphi \in \T$, and for $\kappa(\varphi)=\frac{\phi'(\varphi+\alpha)}{\phi'(\varphi)}$ as in \eqref{first conj bbis}, we see that
$$
\begin{pmatrix}
\phi'(\varphi+\alpha) & 0 \\
0 & -\phi'(\varphi+\alpha)^{-1}
\end{pmatrix}
\begin{pmatrix}
1 & \nu(\varphi) \\
0 & 1
\end{pmatrix}
\begin{pmatrix}
\phi'(\varphi)^{-1} & 0 \\
0 & -\phi'(\varphi)
\end{pmatrix}=\begin{pmatrix}
\kappa(\varphi) & 1\\
0 & \kappa(\varphi)^{-1}
\end{pmatrix},
$$
with $\nu(\varphi):= -(\phi'(\varphi)\phi'(\varphi+\alpha))^{-1}<0$.
  
For all $\varphi\in \T$, we let  $Z_3(\varphi):=Z_2(\varphi) \cdot \mathrm{diag}(\phi'(\varphi),-(\phi')^{-1}(\varphi))$, such that $Z_3 \in C^\omega(\T, \mathrm{SL}(2,\R))$. By  \eqref{first conj bbis}, we thus get
$$
Z_3(\varphi+\alpha)^{-1} S_0^{V}(\varphi) Z_3(\varphi)=\begin{pmatrix}
1 & \nu(\varphi)\\
0 & 1
\end{pmatrix}.
$$
Set $\nu_0:=\int_\T \nu(\varphi)\, d\varphi<0$, so that $\nu-\nu_0\in C_0^\omega(\T,\R)$. Since   $\mathfrak{B}(\alpha)<+\infty$ and   $\nu$ is analytic,  the following cohomological equation has  a solution $\mu\in C^\omega( \T , \R)$:
$$
\mu(\varphi+\alpha)-\mu(\varphi)=\nu(\varphi)-\nu_0,\quad \forall\, \varphi\in \T.
$$
We set $Z:=Z_3 \begin{pmatrix}
1 & \mu\\
0 & 1
\end{pmatrix}\in C^\omega(\T,\mathrm{SL}(2,\R))$. By the successive definitions of $Z_1,Z_2,Z_3$, and by Lemma \ref{prop beta}, for $\varphi\in \T$, we obtain the following expression of the matrix $Z(\varphi)$:
\begin{align}
Z(\varphi)&=\begin{pmatrix}
\phi'(\varphi) & \mu(\varphi) \phi'(\varphi)\\
(\phi(\varphi)-\gamma\circ \phi(\varphi))' & (\phi'(\varphi))^{-1} + \mu(\varphi) (\phi(\varphi)-\gamma\circ \phi(\varphi))'
\end{pmatrix}\nonumber\\
&=\begin{pmatrix}
	\phi'(\varphi) & \mu(\varphi) \phi'(\varphi)\\
	\phi'(\varphi-\alpha) & (\phi'(\varphi))^{-1} + \mu(\varphi) \phi'(\varphi-\alpha) 
\end{pmatrix}. \label{prop conjug def}
\end{align}

As $\phi$ is a circle diffeomorphism, the first coefficient in the matrix does not vanish, and hence the conjugacy map $Z$ is homotopic to the identity. Moreover, for all $\varphi \in \T$, we have
$$
Z(\varphi+\alpha)^{-1}S_0^{V}(\varphi)Z(\varphi)=\begin{pmatrix}
1 & \nu(\varphi)-\mu(\varphi+\alpha)+\mu(\varphi)\\
0 & 1
\end{pmatrix}=\begin{pmatrix}
1 & \nu_0\\
0 & 1
\end{pmatrix}.
$$
\end{proof}

\subsection{Pure point spectrum of dual Schr\"odinger operators}\label{peure poinnt}

Since the reduction to the parabolic cocycle $(\alpha,B_0)$ is a key point in the proof of our Main Theorem, we shall provide another proof of this fact based on dual Schr\"odinger operators and Aubry duality. In general, Aubry duality is based on the fact that the localization properties of the dual Schr\"odinger operators can be used to show that certain Schr\"odinger cocycles are reducible. One may consult \cite{AvilaJito} for more references.  

Let  $(\widehat{v}_n)_{n \in \Z}$ be the Fourier coefficients of the analytic  
potential $V:=-f'\circ \phi-2$, i.e.,  $V\colon \varphi\mapsto\sum_{n\in \Z} \widehat{v}_n e^{2 \pi \mathrm{i} n \varphi}$. For any phase $\varphi_0\in \T$, we define the \textit{dual Schr\"odinger operator} $\widehat{H}_{V,\alpha,\varphi_0}$. It acts on $\widehat u=(\widehat{u}_n)_{n\in \Z} \in \ell^2(\Z)$ in the following way:
\begin{equation}\label{defi dual op}
(\widehat{H}_{V,\alpha,\varphi_0}(\widehat{u}))_n:=\sum_{k \in \Z} \widehat{v}_{n-k} \widehat{u}_k + 2\cos(2\pi(\varphi_0+n\alpha))\widehat{u}_n,\quad \forall\, n \in \Z.
\end{equation}

Let us denote by $\widehat{\phi}'=(\widehat{\phi}'_n)_{n \in \Z}$ the Fourier coefficients of the function $\phi'\in C^\omega(\T , \R)$. In other words, $\phi\colon \varphi\mapsto\sum_{n\in\Z} \widehat{\phi}_n e^{2 \pi \mathrm{i} n \varphi}$, with $\widehat{\phi}'_n= 2 \pi \mathrm{i} n \widehat{\phi}_n$, for   $n \in \Z$. 

\begin{lemma}\label{prop dualll}
We have
\begin{equation}\label{eq v potentiel}
V(\varphi)=
-\frac{\phi'(\varphi+\alpha)+\phi'(\varphi-\alpha)}{\phi'(\varphi)},\quad \forall\, \varphi \in \T,
\end{equation}
which yields
\begin{equation}\label{eq varepsilon dual}
\widehat{H}_{V,\alpha,0} (\widehat{\phi'})=0.
\end{equation}
By analyticity, the sequence $(\widehat{\phi}'_n)_{n \in \Z}$ decays exponentially fast.  In particular, we have $(\widehat{\phi}'_n)_{n\in \Z}\in \ell^2(\Z)$, and the energy $E=0$ is in the point spectrum $\Sigma_{\mathrm{pp}}(\widehat{H}_{V,\alpha,0})$ of the dual operator $\widehat{H}_{V,\alpha,0}$. 
\end{lemma}

\begin{proof}
By the second identity obtained  in \eqref{deq fgammama},  for all $\varphi \in \T$, it holds
\begin{equation*}
f\circ \phi(\varphi)
=\phi(\varphi+\alpha)-2\phi(\varphi)+\phi(\varphi-\alpha),
\end{equation*}
and then, by taking the derivative of the previous expression, we get
\begin{equation*}
f'\circ \phi(\varphi)\cdot  \phi'(\varphi)=\phi'(\varphi+\alpha)-2 \phi'(\varphi)+\phi'(\varphi-\alpha).
\end{equation*}
Since $V=-f'\circ \phi-2$, this can also be rewritten as
\begin{equation}\label{der eq}
V(\varphi) \cdot \phi'(\varphi) + \phi'(\varphi+\alpha)+\phi'(\varphi-\alpha)=0,
\end{equation}
which gives \eqref{eq v potentiel}.

Therefore, in Fourier series, \eqref{der eq} yields
$$
\sum_{n\in \Z} \sum_{k\in \Z} \widehat{v}_{n-k}\widehat{\phi}'_k  e^{2 \pi \mathrm{i} n \varphi}+\sum_{n\in \Z} \widehat{\phi}'_n(e^{2 \pi \mathrm{i} n\alpha}+e^{-2 \pi \mathrm{i} n\alpha})e^{2 \pi \mathrm{i} n \varphi}=0.
$$
Equivalently, we have
$$
(\widehat{H}_{V,\alpha,0}(\widehat{\phi'}))_n=\sum_{k\in \Z} \widehat{v}_{n-k}\widehat{\phi}'_k +2 \cos(2 \pi n \alpha)\widehat{\phi}'_n=0,\quad \forall\, n \in \Z,
$$
which concludes the proof. 
\end{proof}

As in classical Aubry duality, we may then define a Bloch wave $U\colon \varphi \mapsto \begin{pmatrix}
\phi'(\varphi)\\
\phi'(\varphi-\alpha)
\end{pmatrix}$. It provides an invariant section, in the following sense:\footnote{This can also be checked directly using  \eqref{der eq}.}
\begin{equation}\label{invariant u function}
S_0^{V} (\varphi)\cdot U(\varphi)= U(\varphi+\alpha),\quad \forall\,   \varphi \in \T.
\end{equation}
In our case, the phase is equal to $0$ and the frequency $\alpha$ satisfies 
$\mathfrak{B}(\alpha)<+\infty$, hence by point $(2)$ of the precise version of Aubry-duality given in Avila-Jitomirskaya \cite[Theorem 2.5]{AvilaJito}, this provides another proof of the reducibility result that we obtained previously in Proposition \ref{prop reductibilite parab}. 

Note that \eqref{invariant u function} gives another way to see why the connection between existence of analytic invariant curves and parabolic  reducibility of Schr\"odinger  cocycles  can only happen on the edge of the spectrum: indeed, $U$ is associated to the vector field tangent to the invariant curve, hence should have zero degree. But the latter is directly related to the rotation number of the  Schr\"odinger cocycle (as $U$ can be utilized to reduce the  cocycle), which should thus vanish as well; but this corresponds precisely to the right edge of the spectrum.

\subsection{Proof of the Main Theorem}

We remain in the setting of the previous section. Namely, we let $f \in C_0^\omega(\T,\R)$ and assume that the twist map $\psi=\psi_f$ leaves invariant an analytic   graph
$\Gamma_\gamma$, for  some  function $\gamma \in C^\omega(\T,\R)$. We let  $g\colon \varphi \mapsto \varphi+\gamma(\varphi)+f(\varphi)\ \mathrm{mod}\  1$ be the induced diffeomorphism, and    assume that its rotation number $\alpha$   satisfies the Brjuno condition $
\mathfrak{B}(\alpha)< +\infty$. We take $\phi =\phi_\gamma\in C^\omega(\T,\T)$  such that $\phi^{-1} \circ g \circ \phi=r_\alpha$,    set $V:=-f' \circ \phi-2$, and let $B_0 \in \mathrm{SL}(2,\R)$,  $Z \in C^\omega( \T, \mathrm{SL}(2,\R))$ be as in Proposition \ref{prop reductibilite parab}. We also let $H_{V,\alpha,\varphi_0}$ be the Schr\"odinger  operator corresponding to action minimizing trajectories of $\psi$ on $\Gamma_\gamma$.

In this section we will conclude the proof of our Main Theorem on the existence of a component of absolutely continuous spectrum. The idea of the proof is to use Proposition \ref{prop reductibilite parab} together with the openness of the  \textit{almost reducibility} property proved by Avila in \cite{Avila1}.  Following the arguments of Avila \cite{A1}, it implies the existence of a component of absolutely continuous spectrum.  \\

Let us recall a few concepts which will be useful in the following. 

Given a frequency $\tilde \alpha\in \R\setminus \Q$ and a map $A \in C^\omega(\T,\mathrm{SL}(2,\R))$, the $\mathrm{SL}(2,\R)$-cocycle $(\tilde\alpha, A)$  is called {\it subcritical}  if there exists $\varepsilon>0$ such that the associated Lyapunov exponent satisfies   $L(\tilde \alpha, A(\cdot+{\rm i}\delta))=0$ for any $|\delta|<\varepsilon$. The  cocycle $(\tilde \alpha, A)$ is called {\it almost reducible}   if there exists $\varepsilon>0$ and a sequence $(B^{(n)})_{n \geq 0}$ of maps $B^{(n)}\colon \T \to \mathrm{SL}(2,\R)$ admitting holomorphic extensions to the common strip $\{|\Im z|< \varepsilon\}$ such that $B^{(n)}(\cdot +\tilde  \alpha)^{-1}  A(\cdot) B^{(n)}(\cdot)$ converges to a constant $\mathrm{SL}(2,\R)$-matrix  uniformly in $\{|\Im z|< \varepsilon\}$. Let us recall that by Avila's proof of the almost reducibility conjecture (see  \cite{Avila1,A2,Aglobal}), subcriticality implies almost reducibility.

By \cite{AvilaJito}, almost reducibility is related to the notion of \textit{almost localization} which we now recall. For any $x \in \R$, we set $|x|_\T:=\inf_{j \in \Z}|x-j|$. Fix $\epsilon_0>0$ and $\varphi_0\in \T$. An integer $k\in \Z$ is called an \textit{$\epsilon_0$-resonance} of $\varphi_0$ if $|2\varphi_0 - k\alpha|_{\T}\leq e^{-\epsilon_0 |k|}$ and $|2 \varphi_0 - k\alpha|_{\T}=\min_{|l|\leq |k|} |2\varphi_0 - l\alpha|_{\T}$. 
\begin{defi}[Almost localization]\label{defi almost local}
	Given $\tilde\alpha \in \R\setminus\Q$ and $\tilde V\in C^\omega(\T,\R)$, we say that the family $\{\widehat H_{\tilde V,\tilde \alpha,\varphi}\}_{\varphi\in\T}$ is  \textit{almost localized} if there exist constants $C_0, C_1,\epsilon_0, \epsilon_1>0$ such that for all $\varphi_0\in\T$, any generalized solution $u=(u_k)_{k \in \Z}$ to the eigenvalue problem $\widehat H_{\tilde V,\tilde\alpha,\varphi_0} u = E u$ with $u_0=1$ and $|u_k| \leq 1+|k|$ satisfies
	\begin{equation}\label{almost_localization}
	|u_k| \leq C_1 e^{-\epsilon_1 |k|},\quad \forall\, C_0 |n_j| \leq |k| \leq C_0^{-1} |n_{j+1}|,
	\end{equation}
	where $\{n_j\}_j$  is the set  of  $\epsilon_0$-resonances of $\varphi_0$.
\end{defi} 

By Lemma \ref{lemma spectre} and Proposition \ref{prop reductibilite parab}, our Main Theorem is a consequence of the following result. 

\begin{parab}\label{parab}
	Let $\alpha \in \R\setminus \Q$ 
	and $V \in C^\omega(\T,\R)$. Suppose that for some $E_0 \in \R$ the Schr\"odinger cocycle $(\alpha,S_{E_0}^{V})$ is analytically reducible to a constant parabolic cocycle, i.e., there exist $Z \in C^\omega(\T,\mathrm{SL}(2,\R))$ and $\nu_0\in \R$ such that 
	\begin{equation}\label{parabo red}
	Z(\varphi+\alpha)^{-1}S_{E_0}^{V}(\varphi) Z(\varphi)=B_0=\begin{pmatrix}
	1 & \nu_0\\
	0 & 1
	\end{pmatrix},\quad \forall\, \varphi \in \T.
	\end{equation}
	Assume   that $\nu_0<0$. 
	Then there  exists $\varepsilon_0>0$ such that for all $\varphi_0\in \T$,
	\begin{enumerate}
		\item $\Sigma(H_{V,\alpha,\varphi_0})\cap [E_0-\varepsilon_0,E_0+\varepsilon_0]\subset [E_0-\varepsilon_0,E_0]$;
		\item\label{point deux} for any $E\in \Sigma(H_{V,\alpha,\varphi_0})\cap[E_0-\varepsilon_0,E_0]$, the Schr\"odinger cocycle $(\alpha,S_{E}^{V})$ is almost reducible and subcritical;
		\item\label{point trois} the restriction of the spectral measures  to the interval $[E_0-\varepsilon_0,E_0]$ is absolutely continuous and positive. 
	\end{enumerate}
\end{parab}

Here we state the lemma for $\nu_0<0$ but of course, a symmetric result holds for $\nu_0>0$. 
It is well known (see for instance \cite{Puig1,Puig2}) that reducibility at $E=E_0$ to a parabolic matrix different from the identity   implies that locally on one side of $E_0$ the cocycle $(\alpha,S_E^V)$ will be uniformly hyperbolic, and on another side the fibered rotation number will change monotonically (by strict monotonicity of the second iterate of the Schr\"odinger cocycle  with respect to $E$). In our case, we assume that $\nu_0<0$, hence the cocycle $(\alpha,S_E^V)$ is uniformly hyperbolic for $E \in (E_0,E_0+\varepsilon)$ and its fibered rotation number for $E\in (E_0-\varepsilon,E_0)$ will be strictly larger than at $E=E_0$. This proves the first statement of the $\mathrm{[Parabolicity \Rightarrow ac]\ Lemma}$. 
In the following, we will give the proof of points \eqref{point deux} and \eqref{point trois} in this lemma. \\

We take $\alpha ,V,E_0,Z,B_0,\nu_0$ as in the $\mathrm{[Parabolicity \Rightarrow ac]\ Lemma}$.   
To ease the notation, we assume that $E_0=0$. Point \eqref{point trois} in the $\mathrm{[Parabolicity \Rightarrow ac]\ Lemma}$ follows from the next result.   
\begin{prop}\label{prop subc}
There exists $\varepsilon_1>0$ such that for any $E\in (-\varepsilon_1,\varepsilon_1)$, the Schr\"odinger cocycle $(\alpha,S_{E}^{V})$ is almost reducible, and for any $E\in (-\varepsilon_1,0)\cap \Sigma(H_{V,\alpha,\varphi_0})$, the cocycle $(\alpha,S_{E}^{V})$ is subcritical.  
\end{prop}

\begin{proof}
	By \cite[Corollary 1.3]{Avila1}, almost reducibility is an open property in the set of cocycles $\R \setminus \Q \times C^\omega(\T,\mathrm{SL}(2,\R))$. Hence, we have almost reducibility for all $E \in (-\varepsilon,\varepsilon)$ assuming that $\varepsilon>0$ is sufficiently small. 
	As recalled above, for positive energies $E\in (0,\varepsilon)$, the cocycle $(\alpha,S_E^V)$ is 
	almost reducible to a hyperbolic $\mathrm{SL}(2,\R)$-cocycle (in fact, it is even reducible when $\beta(\alpha)=0$). 
	Besides, for  negative energies $E\in (-\varepsilon,0)\cap\Sigma(H_{V,\alpha,\varphi_0})$, the cocycle $(\alpha,S_{E}^{V})$ is almost reducible to a  parabolic cocycle or a cocycle of rotations.
	
As the potential $V$ and  the conjugacy map $Z$ are analytic, 
 formula \eqref{parabo red} implies that the cocycle $(\alpha,S_0^V)$ is subcritical. By \cite{Aglobal}, subcriticality is also an open property in the spectrum (outside of uniform hyperbolicity), which implies the second statement of Proposition \ref{prop subc}. 
\end{proof}

Given an analytic function $F \in C^\omega(\T,*)$ 
with 
$*=\R$ or $\mathfrak{M}_2(\R)$,
we set $\|F\|_\T:=\sup_{\varphi \in \T} \|F(\varphi)\|$, and for any $h>0$ such that $F$ has a bounded analytic extension to the strip $\{|\Im z|< h\}$, we set $\|F\|_h:=\sup_{|\Im z|<h} \|F(z)\|$. 

In the following, we will give the proof of point \eqref{point trois} in the $\mathrm{[Parabolicity \Rightarrow ac]\ Lemma}$.  Although the result  holds  for any irrational frequency $\alpha\in \R\setminus \Q$, here we present the  proof in the case that $\beta(\alpha)=0$, as it is the one which is related to our problem about the existence of invariant curves.\footnote{We need $\mathfrak{B}(\alpha)<+\infty$ to guarantee the existence of the conjugacy $\phi$ to the rigid rotation by angle $\alpha$.}  We refer to \cite{A1}  for the   case $\beta(\alpha)>0$ (see \cite[pp. 16--20]{A1}). 

It was shown in \cite[Theorem 3.2]{AvilaJito}  that there exist absolute constants $c_0,k_0>0$ with the following property: for any analytic potential $\tilde V \in C^\omega(\T,\R)$ such that  $\|\tilde{V}_\varepsilon\|_{h_0} < c_0 h_0^{k_0}$ for some $h_0 \in (0,1]$,   the family of dual Schr\"odinger  operators $\{ \widehat{H}_{\tilde{V},\alpha,\varphi_0}\}_{\varphi_0 \in \T}$ is almost localized.  
In the following, we let $h_0 \in (0,1)$ be such that the potential $V$ and the conjugacy  map $Z$ have a bounded   analytic extension to the  strip $\{|\Im z|<h_0\}$. 
\begin{lemma}\label{auxiliaire lemme}
	There exists $\varepsilon_0\in (0,\varepsilon_1)$ such that for any $\varepsilon\in (- \varepsilon_0,\varepsilon_0)$, there exist $\tilde{V}_\varepsilon\in C^\omega(\T,\R)$   and $\tilde{Z}_\varepsilon \in C^\omega(\T,\mathrm{SL}(2,\R))$ such that $\|\tilde{V}_\varepsilon\|_{h_0}<c_0 h_0^{k_0}$,
	\begin{equation}\label{conj to  new small pot}
	\tilde{Z}_\varepsilon(\varphi+\alpha)^{-1} S_{\varepsilon}^{V}(\varphi) \tilde{Z}_\varepsilon(\varphi)=S_0^{\tilde{V}_\varepsilon}(\varphi),\quad \forall\, \varphi \in \T,
	\end{equation}
	and such that the family of dual Schr\"odinger  operators $\{ \widehat{H}_{\tilde{V}_\varepsilon,\alpha,\varphi_0}\}_{\varphi_0 \in \T}$ is almost localized.
\end{lemma}


\begin{proof}
	Take $\nu_0 <0$ as above, let $\nu_1:=\sqrt{-\nu_0}>0$, and set $Q_0:=\begin{pmatrix}
	-\frac 12 \nu_1 & -\frac 12 \nu_1\\
	\nu_1^{-1} & -  \nu_1^{-1}
	\end{pmatrix}\in \mathrm{SL}(2,\R)$. The matrix $Q_0$ conjugates the parabolic cocycle $(\alpha,B_0)$  to the Schr\"odinger cocycle $(\alpha,S_0^\mathbf{0})$ associated with a vanishing potential $\mathbf{0}$ and   the energy $E=0$, i.e., $Q_0^{-1} B_0 Q_0=S_0^\mathbf{0}=\begin{pmatrix}
	2 & -1 \\
	1 & 0
	\end{pmatrix}$. Let $Z_0:= Z Q_0 \in C^\omega(\T,\mathrm{SL}(2,\R))$, so that
	\begin{equation*}
		Z_0(\varphi+\alpha)^{-1} S_0^{V}(\varphi) Z_0(\varphi)=S_0^\mathbf{0}=\begin{pmatrix}
			2 & -1 \\
			1 & 0
		\end{pmatrix},\quad \forall\, \varphi \in \T.
	\end{equation*}
	For any $\varepsilon\in \R$, we thus obtain 
	$$
	Z_0(\varphi+\alpha)^{-1} S_{\varepsilon}^{V}(\varphi) Z_0(\varphi)=S_0^\mathbf{0}+ \varepsilon P_1(\varphi),\quad \forall \varphi \in \T,
	$$
	for some analytic map $P_1 \in C^\omega(\T,\mathfrak{M}_2(\R))$.   Then, by   \cite[Lemma 2.2]{AvilaJitoHolder}, 
	for any $\tau>0$,  there exists $\delta=\delta(\tau) >0$ such that  if $|\varepsilon| \cdot \|P_1\|_{h_0}< \delta$, then there exists $\tilde{V}_\varepsilon \in C^\omega(\T,\R)$ such that $\|\tilde{V}_\varepsilon\|_{h_0}<\tau$,   and $\tilde{Z}_\varepsilon \in C^\omega(\T,\mathrm{SL}(2,\R))$ with a bounded analytic extension to the strip $\{|\Im z|<h_0\}$, such that $\|Z_0-\tilde{Z}_\varepsilon\|_{h_0} < \tau$, and 
	$$
	\tilde{Z}_\varepsilon(\varphi+\alpha)^{-1} S_{\varepsilon}^{V}(\varphi) \tilde{Z}_\varepsilon(\varphi)=S_0^{\tilde{V}_\varepsilon}(\varphi),\quad \forall\, \varphi \in \T.
	$$ 
	
	Now, let us take  $\tau=\tau_0:=c_0 h_0^{k_0}$ and  let $\varepsilon_0>0$ be such that  $\varepsilon_0 \cdot \|P_1\|_{h_0}< \delta(\tau_0)$. Then, for any $\varepsilon\in (-\varepsilon_0,\varepsilon_0)$, the  potential $\tilde{V}_\varepsilon$ satisfies  $\|\tilde{V}_\varepsilon\|_{h_0} < c_0 h_0^{k_0}$, 
	hence by the definition of $c_0,k_0$, the family of dual Schr\"odinger  operators $\{ \widehat{H}_{\tilde{V}_\varepsilon,\alpha,\varphi_0}\}_{\varphi_0 \in \T}$ is almost localized.  
\end{proof}

For any $u \in \ell^2(\E)$, any $\varphi_0 \in \T$, we let
$\mu_{V,\alpha,\varphi}^u$ be the \textit{spectral measure}  of $H=H_{V,\alpha,\varphi_0}$ associated to $u$, i.e., such that $
((H-E)^{-1}u,u)=\int_{\R} \frac{1}{E'-E}d\mu_H^u(E')$, for all $E\in\C\setminus\Sigma$.  In what follows, we consider the \textit{canonical spectral measure} corresponding to $u=e_{-1}+e_0$, and denote it by $\mu_{V,\alpha,\varphi}:=\mu_{V,\alpha,\varphi}^{e_{-1}+e_0}$. The support of $\mu_{V,\alpha,\varphi}$ is equal to the spectrum $\Sigma(H)$.  

Let us also recall the definition of  the  {\it integrated density of states} $N_{V,\alpha}\colon\R\to [0,1]$:
$$
N_{V, \alpha}(E):=\int_{\T} \mu_{V,\alpha,\varphi}(-\infty,E] \, d\varphi, \quad \forall\, E \in \R,
$$
where for $i\in \Z$, we let $e_i:=(\delta_{ij})_{j \in \Z}\in \ell^2(\Z)$.

To conclude the proof of the $\mathrm{[Parabolicity \Rightarrow ac]\ Lemma}$, it remains to show the following. 
\begin{prop}\label{main corollaire}
The restriction of the spectral measure $\mu_{V,\alpha,\varphi_0}$ to the interval $[-\varepsilon_0,0]$ is absolutely continuous and positive. 
\end{prop}

Proposition \ref{main corollaire} is proved by repeating the proof of the main result in \cite{A1} in the case $\beta(\alpha)=0$ (see  \cite[p. 15]{A1}).  We  denote by $\mathcal{B}$   the set of energies $E \in \R$ such that the iterates $(k\alpha,A_k)=(\alpha,S_E^{V})^k$ of the cocycle $(\alpha,S_E^{V})$ are uniformly bounded, i.e., 
$$
\sup_{k\geq 0} \|A_k\|_\T<+\infty.
$$ 
Let us recall the following classical characterization (see \cite{Gilbert}). 
\begin{theorem}\label{theoreme 2.4}
The restriction $\mu_{V,\alpha,\varphi_0}|_{\mathcal{B}}$ is absolutely continuous for all $\varphi_0 \in \R$. 
\end{theorem}
In fact, Theorem \ref{theoreme 2.4} is implied by the following very general estimate contained in \cite{JL1,JL2} and
 \cite{A1}. 
\begin{lemma}[Lemma 2.5 in \cite{A1}]\label{lemme 2.5}
For all $\varphi_0 \in \T$, we have 
$$
\mu_{V,\alpha,\varphi_0}(E-  \epsilon,E+ \epsilon) \leq C \epsilon \sup_{0\leq k \leq C\epsilon^{-1}} \|A_k\|_\T^2,
$$ 
for some universal constant $C>0$. 
\end{lemma}

\begin{proof}[Proof of Proposition \ref{main corollaire}]
Let $\Sigma_{\varepsilon_0}:=\Sigma(H_{V,\alpha,\varphi_0})\cap [-\varepsilon_0,0]$, let $\mathcal{B}$ be the set of energies $E\in \Sigma_{\varepsilon_0}$ such that the cocycle $(\alpha,S_E^{V})$ is bounded, and let $\mathcal{R}$ be the set of energies $E\in \Sigma_{\varepsilon_0}$ such that  $(\alpha,S_E^{V})$ is reducible. By Theorem \ref{theoreme 2.4} recalled above, it is sufficient to prove that for every $\varphi_0 \in \T$, the canonical spectral measure $\mu=\mu_{V,\alpha,\varphi_0}$ satisfies $\mu(\Sigma_{\varepsilon_0}\setminus\mathcal{B})=0$. 

As noted in \cite{A1}, for any $E \in\mathcal{R} \setminus\mathcal{B}$, we have $N_{V,\alpha}(E)\in\Z \oplus \alpha \Z $, and $(\alpha,S_E^{V})$ is analytically reducible to a parabolic cocycle; in particular, $\mathcal{R} \setminus\mathcal{B}$ is countable. Moreover, there are no eigenvalues in $\mathcal{R}$ (if $H_{V,\alpha,\varphi_0}u=Eu$ with $E \in\mathcal{R}$ and $u\neq 0$, then $\inf_{n \in \Z} |u_n|^2+|u_{n+1}|^2>0$ hence $u \notin\ell^2(\Z)$), and then, $\mu(\mathcal{R} \setminus\mathcal{B})=0$. Therefore, it is enough to prove that $\mu(\Sigma_{\varepsilon_0}\setminus\mathcal{R})=0$. 

By \eqref{conj to  new small pot}, for any energy $\varepsilon\in (-\varepsilon_0,\varepsilon_0)$, the cocycle  $(\alpha,S_{\varepsilon}^{V})$ is (almost) reducible if and only the cocycle $(\alpha,S_0^{\tilde{V}_\varepsilon})$ is. 
Moreover, by Theorem 2.5 and Theorem 4.1 in \cite{AvilaJito}, (almost) reducibility   of $(\alpha,S_0^{\tilde{V}_\varepsilon})$ is related to the (almost) localization  of the family of dual Schr\"odinger operators  $\{\widehat{H}_{\tilde{V}_\varepsilon,\alpha,\varphi}\}_{\varphi \in \T}$. We know by   \cite[Theorem 3.3]{AvilaJito} that there exist a phase $\varphi_0 \in \T$ and a sequence $\widehat u=(\widehat u_j)_{j \in \Z}$,  with $ \widehat u_0=1$ and $|\widehat u_j|\leq 1$ for all $j\in \Z$, such that $\widehat{H}_{\tilde{V}_\varepsilon,\alpha,\varphi_0}  \widehat u=0$. As explained in \cite{AvilaJito}, $\widehat u$ can be utilized to (almost) reduce the cocycle $(\alpha,S_0^{\tilde{V}_\varepsilon})$. By almost localization of $\widehat{H}_{\tilde{V}_\varepsilon,\alpha,\varphi_0}$, $\widehat u$ decays exponentially fast between the resonances as in \eqref{almost_localization}. If $\varphi_0$ is not resonant, then by  \cite[Remark 3.3]{AvilaJito}, the sequence $(\widehat u_j)_{j \in \Z}$ decays exponentially fast and $(\alpha,S_0^{\tilde{V}_\varepsilon})$  is actually reducible.  In the following, we  consider the case where we have resonances. 

Following \cite{A1}, for any  integer $m \geq 0$, we let $K_m\subset \Sigma_{\varepsilon_0}$ be the set  of energies $E$ such that for some $\varphi_0 \in \T$, the dual operator $\widehat{H}=\widehat{H}_{\tilde{V}_E,\alpha,\varphi_0}$ has a bounded normalized solution $\widehat{H} \widehat u=0$ with a resonance $2^m \leq |n_j| < 2^{m+1}$. By the Borel-Cantelli lemma, to conclude the proof, it is enough to show that $\sum_m \mu(\overline{K_m})<+\infty$. Indeed,  by Theorem 3.3 in \cite{A1}, we have $\Sigma_{\varepsilon_0}\setminus\mathcal{R} \subset \limsup_m K_m$, and then, $\sum_m \mu(\overline{K_m})<+\infty$ implies that $\mu(\Sigma_{\varepsilon_0}\setminus\mathcal{R}) \leq \mu(\limsup_m K_m)=0$. 

By  \cite[Theorem 3.8]{A1}, there exist constants $C_1,c_1>0$ such that for each integer $m \geq 0$ and each energy $E \in K_m$, there exists an open neighborhood $J_m(E)$ of $E$  of size $\epsilon_m:=C_1e^{-c_1 2^m}$ so that the iterates of $(\alpha,A)=(\alpha,S_E^{V})$ satisfy $\sup_{0 \leq k \leq 10 \epsilon_m^{-1}}\|A_k\|_\T \leq e^{o(2^m)}$. By Lemma \ref{lemme 2.5}, we thus get
\begin{equation}\label{estimee mesure jm}
\mu(J_m(E)) \leq Ce^{o(2^m)}|J_m(E)|,
\end{equation}
where $|\cdot|$ is the Lebesgue measure. Take a finite subcover $\overline{K_m}\subset \cup_{j=0}^{r_m} J_m(E_j)$ such that every $x \in \R$ is contained in at most $2$ different $J_m(E_j)$. 

By  \cite[Lemma 3.11]{A1}, the integrated density of states $N_{V,\alpha}$ satisfies   
$|N_{V,\alpha}(J_m(E))|\geq  c |J_m(E)|^2$, 
for some constant $c >0$. Besides, by 
\cite[ Lemma 3.13]{A1}, there exist  constants  $C_2,c_2>0$ such that for any $E \in K_m$, it holds $\|N_{V,\alpha}(E)-k\alpha\|_{\T} \leq C_2 e^{-c_2 2^{m}}$ for some integer $k$ with $|k|< C_2 2^m$. Therefore, the set $N_{V,\alpha}(K_m)$ can be covered by $C_2 2^{m+1}$ intervals $(T^k_m)_{0 \leq k\leq C_2 2^{m+1}}$ of length $C_2 e^{-c_2  2^{m}}$.  
For some constant  $C_3>0$, we have  $|T^k_m|< C_3 |N_{V,\alpha}(J_m(E))|$, for any integers $m \geq 0$,  $0 \leq k\leq C_2 2^{m+1}$, and any energy $E \in K_m$. Hence,  for a given interval $T^k_m$, there are at most $2C_3+4$ intervals $J_m(E_j)$ such that $N_{V,\alpha}(J_m(E_j))$ intersects $T^k_m$. We deduce that for each  integer $m \geq 0$, it holds $r_m \leq C_4 2^m$, with $C_4:=4C_2(C_3+2)$. Then, by \eqref{estimee mesure jm}, we obtain
$$
\mu(K_m) \leq \sum_{j=0}^{r_m} \mu(J_m(E_j)) \leq C_4 2^{m} \cdot C e^{o(2^m)}\cdot C_1 e^{-c_1 2^m}=O\big(e^{-c_1 2^{m-1}}\big),
$$
which gives $\sum_m \mu(\overline{K_m}) < +\infty$, and concludes the proof of Proposition \ref{main corollaire}. 
\end{proof}

As a consequence of point \eqref{point deux} in the $\mathrm{[Parabolicity \Rightarrow ac]\ Lemma}$  and of \cite[Theorem 7.1]{LYZZ} for a frequency $\alpha$ satisfying   the weak Diophantine condition $\beta(\alpha)=0$  and in the subcritical regime, we also get some result about the homogeneity of the spectrum near its right edge:
\begin{prop}
	Under the same assumptions as in the $\mathrm{[Parabolicity \Rightarrow ac]\ Lemma}$, if $\alpha$ moreover satisfies $\beta(\alpha)=0$, then there exists $\kappa>0$ such that 
	$| (E-\varepsilon,E+\varepsilon) \cap \Sigma(H_{V,\alpha,\varphi_0})| > \kappa\varepsilon$, for all $E\in \Sigma(H_{V,\alpha,\varphi_0})\cap (E_0-\varepsilon_0,E_0)$, and for all $0<\varepsilon<E_0-E$.
\end{prop}

\section{Concluding remarks}

The main result of this paper is the existence of a component of absolutely continuous spectrum whenever there exists an analytic invariant curve. This is a semi-global result, which does not require explicitely the smallness of the potential. We finish this paper with several conjectures and questions related to the effect that the transitions   from KAM to weak KAM regime has on the spectrum of the corresponding Schr\"{o}dinger operators. 

As in the introduction we consider a one-parameter family of twist maps $\{\psi_{\lambda f}\}_{\lambda \in \R}$. For a fixed typical rotation number $\alpha$ it is believed that 
\begin{enumerate}
	\item for $0 \leq \lambda < \lambda_{cr}(\alpha)$, there exists a smooth invariant curve such that the restricted dynamics is conjugated to a rigid rotation by $\alpha$; 
	\item for $\lambda=\lambda_{cr}(\alpha)$, the invariant curve still exists but it is not analytic anymore (the \textit{critical curve}, possibly, is only $C^{1+\epsilon}$-smooth);
	\item for $\lambda>\lambda_{cr}(\alpha)$, there exists an invariant cantori with rotation number $\alpha$. The dynamics on the cantori is hyperbolic. 
\end{enumerate}

Below, we discuss how the above dynamical transition from elliptic dynamics (KAM) to hyperbolic dynamics (weak KAM) may  affect the spectral type of the corresponding Schr\"{o}dinger operators.

 It is likely that in the weak KAM  regime $\lambda>\lambda_{cr}
(\alpha)$ there will be no absolutely continuous spectrum. 
This case corresponds to rather rough discontinuous potentials. The case where the potential has just one jump discontinuity
was considered in \cite{DamanikKillip}. It was shown there that the spectrum has no absolutely continuous component. The proof is based on the 
non-deterministic argument and Kotani approach which imply that the set of energies at which the Lyapunov exponent vanishes has zero 
Lebesgue measure. In the weak KAM case the Kotani theorem should be extended in the direction of asymptotic non-determinism.
It follows from the hyperbolicity of the dynamics that one can find two different sequences $\{V^{1,2}_0(n), \, n \in \Z\}$ of values of potentials 
which exponentially converge to each other as $n \to \pm \infty$. It is natural to ask whether the Kotani argument can be extended to this case.

We have shown that an absolutely continuous component exists for $\lambda<\lambda_{cr}(\alpha)$. It is an interesting question 
whether the spectrum is mixed.  Since dynamical properties are related only to the edge of the  spectrum it is natural to expect coexistence of absolutely continuous spectrum and point spectrum for $0\leq \lambda_{cr}(\alpha)-\lambda
\ll 1$. 

As was said above we do not expect  absolutely continuous spectrum in the supercritical case $\lambda>\lambda_{cr}(\alpha)$.
Again one can ask whether the spectrum is mixed in this case. Notice that for $\lambda>\lambda_{cr}(\alpha)$ the edge of the spectrum is
expected to move to the left. Namely, $\Sigma(\mathcal{H})\subset (-\infty,E(\lambda)]$  with $E(\lambda)<0$, and $E(\lambda) \to 0$ as $\lambda \to \lambda_{cr}(\alpha)$.

Finally, it is natural to ask what is the spectral type near the edge of the spectrum $E=0$  in the critical case $\lambda=\lambda_{cr}(\alpha)$.
It is tempting to think that there a singular continuous component is created.

%
%
%


\begin{thebibliography}{99}

\bibitem{Aubry} Aubry, S.; \textit{The twist map, the extended Frenkel-Kontorova model and the devil's staircase}, Physica D: Nonlinear Phenomena, \textbf{7(1-3)} (1983), pp. 240--258.

\bibitem{AubryDaeron}{Aubry, S. \& Le Daeron, P.Y.; \textit{The discrete Frenkel-Kontorova model and its extensions: I. Exact results for the ground-states}, Physica D: Nonlinear Phenomena, \textbf{8(3)} (1983), pp. 381--422.}

\bibitem{Avila1}{Avila, A.; \textit{Almost reducibility and absolute continuity I}, preprint.}

\bibitem{A1}
{Avila, A.; \textit{The absolutely continuous spectrum of the almost Mathieu operator}, preprint.}


\bibitem{A2}
Avila, A.; \textit{KAM, Lyapunov exponents and the spectral dichotomy for one-frequency Schr\"{o}dinger operators}, preprint.


\bibitem{Aglobal}
Avila, A.; \textit{Global theory of one-frequency Schr\"{o}dinger operators}, Acta Math., {\bf 215} (2015), pp. 1--54.

\bibitem{AvilaJito}{Avila, A.; Jitomirskaya, S.;  \textit{Almost localization and almost reducibility},  Journal of the European Mathematical Society \textbf{12} (2010), pp. 93--131.}

\bibitem{AvilaJitoHolder}
Avila, A.; Jitomirskaya, S.; \textit{H\"{o}lder continuity of absolutely continuous spectral measures for one-frequency Schr\"{o}dinger operators}, Commun. Math. Phys., {\bf 301} (2011), pp. 563--581.

\bibitem{Birkhoff}
Birkhoff, G.D.; \textit{Surface transformations and their dynamical application}, Acta Math. \textbf{43} (1920), pp.  1--119.

\bibitem{Damanik}
Damanik, D.; \textit{Schr\"odinger operators with dynamically defined potentials},  Ergodic Theory and Dynamical Systems, \textbf{37(6)} (2017), pp. 1681--1764.

\bibitem{DamanikKillip}
Damanik, D.; Killip, R.; \textit{Ergodic potentials with a discontinuous sampling function are non-deterministic}, Math. Res. Lett. \textbf{12} (2005), pp. 187--192. 

\bibitem{Eliasson} 
Eliasson, H.; \textit{Floquet Solutions for the $1$-Dimensional
Quasi-Periodic Schr\"odinger Equation}, Commun. Math. Phys. \textbf{146}  (1992), pp. 447--482.

\bibitem{EFK} Eliasson, H.; Fayad, B.; Krikorian, R.; \textit{Around the stability of KAM tori}, Duke Mathematical Journal, \textbf{164(9)}  (2015), pp. 1733--1775.

\bibitem{ElFaKr} Eliasson, H.; Fayad, B.; Krikorian, R.; \textit{Jean-Christophe Yoccoz and the theory of circle diffeomorphisms},  arXiv preprint arXiv:1810.07107 (2018).

\bibitem{EKMY} Eliasson, H.;  Kuksin, S.; Marmi, S.; Yoccoz, J. C.; \textit{Dynamical Systems and Small Divisors: Lectures given at the CIME Summer School held in Cetraro Italy}, June \textbf{13-20}  (1998), Springer, 2004.

\bibitem{Fathi} Fathi, A.; \textit{Une approche plus topologique de la d\'emonstration du th\'eor\`eme de Birkhoff}, appendice au chapitre 1 de \cite{Herman}, pp. 39--46.

\bibitem{Gentile}{Gentile, G.; \textit{Invariant curves for exact symplectic twist maps of the cylinder with Bryuno rotation numbers}, Nonlinearity \textbf{28.7} (2015): 2555.}

\bibitem{Gilbert} Gilbert, D.J.; \textit{On subordinacy and analysis of the spectrum of one-dimensional
Schr\"odinger operators}, J. Math. Anal. Appl. 128 (1987), 30-56.


\bibitem{Herman} Herman, M.; \textit{Sur les courbes invariantes par les diff\'eomorphismes de l'anneau}, Vol. \textbf{1}, Ast\'erisque (1983), pp. 103--104. 

\bibitem{JL1} Jitomirskaya, S.; Last, Y.; \textit{Power-law subordinacy and singular spectra. I. Half- line operators}, Acta Math. \textbf{183} (1999), no. 2, pp. 171--189.

\bibitem{JL2}  Jitomirskaya, S.; Last, Y.; \textit{Power law subordinacy and singular spectra. II. Line operators}, Comm. Math. Phys. \textbf{211} (2000), no. 3, pp. 643--658.

\bibitem{KS} Khanin, K. M.; Y. G., Sinai; \textit{Renormalization group method in the theory of dynamical systems}, International Journal of Modern Physics B  Vol. \textbf{2}, No. 2 (1988), pp. 147--165.

\bibitem{Kot} Kotani, S.; \textit{Jacobi matrices with random potentials taking finitely many values}, Rev. Math.
Phys. \textbf{1} (1989), pp. 129--133. 

\bibitem{LYZZ} Leguil, M.; You, J.; Zhao, Z.; Zhou, Q.; \textit{Asymptotics of spectral gaps of quasi-periodic Schr\"odinger operators}, arXiv preprint, \url{https://arxiv.org/abs/1712.04700.}

\bibitem{Mather} Mather, J.; \textit{A criterion for the non-existence of invariant circles}, Publications Math\'ematiques de l'Institut des Hautes \'Etudes Scientifiques, \textbf{63(1)} (1986), pp. 153--204.

\bibitem{Puig1} Puig, J.; \textit{A nonperturbative Eliasson’s reducibility theorem}, Nonlinearity \textbf{19} (2006), pp. 
355--376.

\bibitem{Puig2} Puig, J.; \textit{Cantor Spectrum for quasi-periodic Schr\"odinger  Operators}, Mathematical
Physics of Quantum Mechanics. Springer Lecture Notes in Physics, Vol. \textbf{690} (2006).
Asch, Joachim; Joye, Alain (Eds.). ISBN: 3-540-31026-6.

\bibitem{Yoccoz} Yoccoz, J.-C., \textit{Analytic linearization of circle diffeomorphisms}, Dynamical systems and small divisors (Cetraro, 1998),  Lecture Notes in Math., \textbf{1784}, Fond. CIME/CIME Found. Subser., Springer, Berlin (2002), pp. 125--173. 
\end{thebibliography}
\end{document}